\documentclass[a4paper,12pt]{article}
\usepackage{amsmath,amssymb,amsthm}
\usepackage{comment}
\usepackage[pdftex]{graphicx}
\usepackage{pgfplots} 
\usepackage{float}
\usepackage{pdfpages}

\numberwithin{equation}{section}
\usepackage[margin=20mm]{geometry}

\newtheorem{thm}{Theorem}[section]

\newtheorem{cor}[thm]{Corollary}
\newtheorem{exam}[thm]{Example}
\newtheorem{rem}[thm]{Remark}
\newtheorem{lem}[thm]{Lemma}

\newtheorem{assum}[thm]{Assumption}

\makeatletter

\begin{document}
\title{Remarks on the limiting behaviors of \\ generalized elephant random walks}
\author{Yuichi Shiozawa\thanks{Department of Mathematics,
Graduate School of Science,
Osaka University, Toyonaka, Osaka, 560-0043,
Japan; \texttt{shiozawa@math.sci.osaka-u.ac.jp}}}
\maketitle

\begin{abstract}
We study the limiting behaviors of a generalized elephant random walk on the integer lattice. 
This random walk is defined by using two sequences of parameters expressing 
the memory at each step from the whole past  
and the drift of each step to the right, respectively. 
This model is also regarded as a dependent Bernoulli process. 
Our results reveal how the scaling factors are determined by 
the behaviors of the parameters.   
In particular, we allow the degeneracy of the parameters.  
We further present several examples in which 
the scaling factors are explicitly computed. 
\end{abstract}

\section{Introduction}
We are concerned with the limiting behaviors of 
a generalized elephant random walk on the integer lattice. 
This random walk is defined by using the two sequences of parameters expressing 
the memory at each step from the whole past 
and the drift of each step to the right, respectively.  
Our purpose in this paper is to reveal  how the behaviors of these two sequences  
would determine the scaling factors of the random walk under consideration.

Drezner and Farnum \cite{DF93} introduced a generalized binomial distribution 
associated with the correlated Bernoulli sequence, 
and calculated the moments and variance. 
Heyde \cite[Theorem 1]{H04} utilized the martingale theory (see, e.g., \cite{HH80}) 
to show the existence of the phase transition on the Gaussian fluctuation around the mean 
in terms of the correlation parameter. 
The approach of Heyde \cite{H04} is further developed and applied to  
the proofs of several limit theorems 
for more general dependent Bernoulli sequences 
(see, e.g., \cite{JJQ08,MMY20,WQY12,ZZ15}). 
As we see from the previous works above, if the Bernoulli sequence does not degenerate, 
then the almost sure linear scaling limit exists for the associated process.

On the other hand, Sch\"utz and Trimper \cite{ST04} introduced 
the model of an elephant random walks in which 
the law of each step is given by the steps until just before 
with the independent random drift.
They observed the existence of the phase transition on the diffusivity 
by revealing the asymptotic behavior of the second moment. 
Coletti, Gava and Sch\"utz \cite{CGS17} and Bercu \cite{B18} 
also applied the martingale theory to establish the limit theorems as Heyde \cite{H04}. 
For the superdiffusive elephant random walk with bias,  
Kubota and Takei \cite{KT19} showed that if the bias decays, 
then the $L^2$-scaling limit exhibits the phase transition, 
together with the explicit scaling factor.
When the variance does not degenerate, 
they also revealed the Gaussian fluctuations around both the mean and the random drift. 
As pointed in \cite{KT19}, a class of  elephant random walks corresponds 
to the dependent Bernoulli sequences by simple relations. 
For the recent development on elephant random walks, see, e.g., 
\cite{BL21,B22+,FHM21,HOT22+}. 

Motivated by \cite{KT19}, we have two preliminary questions as follows:
\begin{enumerate}
\item[(i)] Almost sure scaling limit  of the elephant random walk with decaying bias;
\item[(ii)] Gaussian fluctuations of the elephant random walk with degenerate variance.
\end{enumerate}
In this paper, we address these two questions for 
generalized elephant random walks. 
More precisely, for (i), we establish the almost sure scaling limit in Corollary \ref{cor:lln-0}. 
As we see from examples in Section \ref{sect:lln}, 
the scaling factor is not necessarily linear and explicitly affected 
by the convergence rates of the sequences of the parameters. 
We further see that the $L^2$-scaling limit of \cite[Theorem 4]{KT19} remains true 
as the almost sure convergence (Remark \ref{rem:as-conv}). 
For (ii), we find suitable scaling factors for the Gaussian fluctuations 
by taking into consideration the degeneracy of the variance 
in Theorems \ref{thm:quad-var} and \ref{thm:fluct}.  
As its application, we have the scaling limit result 
interpolating Corollary \ref{cor:lln-0} and Theorem \ref{thm:conv-limit} (Corollary \ref{cor:conv-limit-1}).  

Our approach here is based on the martingale theory as in the previous works 
(\cite{B18,CGS17,H04,JJQ08,KT19,MMY20,WQY12,ZZ15}). 
For (i), 
Corollary \ref{cor:lln-0} follows from the law of the large numbers (Theorem \ref{thm:lln-0}). 
To prove Theorem \ref{thm:lln-0}, 
we calculate the quadratic variation of some martingale associated with  
the scaling factor. 
This approach is a simple modification of 
\cite[Theorem 1]{CGS17}, \cite[Theorem 2.1]{JJQ08}, \cite[Theorem 1]{KT19} 
and \cite[Theorem 2]{WQY12}.
Here we would like to mention that, as far as the author knows, 
no other results are available about the almost sure scaling limit except for the linear scaling.
For (ii), we show the Gaussian fluctuations (Theorems \ref{thm:quad-var} (1) and \ref{thm:fluct}) 
by getting the asymptotic behavior of the quadratic variation 
of a square integrable martingale $\{M_n\}_{n=1}^{\infty}$ 
(Lemma \ref{lem:moment} (3) and \eqref{eq:quad-asymp}).  
To do so, we obtain the exact decay rate for the linear scaling limit 
as an application of the law of the large numbers 
(Lemma \ref{lem:lln-order} (2)). 

Related to our preliminary questions (i) and (ii), 
we would like to mention that 
Konno \cite{K09} clarified a structural similarity  
between correlated random walks and quantum walks. 
Using this similarity, he further calculated the characteristic function 
and proved the convergence in distribution for the correlated random walk 
even if the correlation parameter degenerates. 

The rest of this paper is organized as follows:
In Section \ref{sect:pre}, we first introduce 
a model of generalized elephant random walks.  
We then present the moment formula and elementary calculations 
which are necessary for the subsequent sections. 
In Section \ref{sect:lln}, we first prove the law of the large numbers, 
almost sure and also $L^2$-scaling limits of the random walks. 
We then compute the scaling factors in examples.  
In Section \ref{sect:gauss}, we establish the Gaussian fluctuations 
around the mean and the random drift by proving the central limit theorems 
and the laws of the iterated logarithm. 
In connection with Section \ref{sect:lln}, 
we also discuss again the scaling limit problem. 
We then compute scaling factors related to the Gaussian fluctuations. 
Appendix \ref{sect:append} includes the calculations of series and 
the computations of the scaling factors 
in the examples of Sections \ref{sect:lln} and \ref{sect:gauss}.

\section{Preliminaries}\label{sect:pre}
We first introduce a model of generalized elephant random walks.  
Let $q\in [0,1]$, and let $\{\alpha_n\}_{n=1}^{\infty}$ and $\{\varepsilon_n\}_{n=1}^{\infty}$ 
be $[0,1]$-valued sequences. 
Let $\{X_n\}_{n=1}^{\infty}$ be a sequence of $\{-1,1\}$-valued random variables 
defined on a probability space $(\Omega,{\cal F}, P)$ 
such that 
$$P(X_1=1)=q, \quad P(X_1=-1)=1-q$$ 
and 
\begin{equation}\label{eq:dist}
P(X_{n+1}=\pm 1\mid X_1,\dots,X_n)
=\alpha_n\frac{\sharp\{i=1,\dots, n \mid X_i=\pm 1\}}{n}
+(1-\alpha_n)\frac{1\pm \varepsilon_n}{2} \quad (n\geq 1).
\end{equation}
We then define $S_0=0$ and $S_n=X_1+\dots+X_n \ (n\geq 1)$. 
We call $\{S_n\}_{n=0}^{\infty}$ a {\it generalized elephant random walk} 
with correlation parameters $\{\alpha_n\}_{n=1}^{\infty}$ and drift parameters $\{\varepsilon_n\}_{n=1}^{\infty}$.
By \eqref{eq:dist}, 
the $(n+1)$st step $X_{n+1}$ of $\{S_n\}_{n=0}^{\infty}$ 
follows a uniformly chosen past step among $X_1,\dots, X_n$ with probability $\alpha_n$, 
and the random walk with drift $\varepsilon_n$ with probability $1-\alpha_n$.   
If $\alpha_n=\alpha \ (n\geq 1)$ for some $\alpha\in [0,1]$, 
then $\{S_n\}_{n=0}^{\infty}$ is the so-called elephant random walk with bias (\cite{B18,CGS17,KT19,ST04}).

Let $X_n'=(1+X_n)/2$. If we take $\varepsilon_n=2q-1 \ (n\geq 1)$, 
then $\{X_n'\}_{n=1}^{\infty}$ is a sequence of the correlated Bernoulli random variables 
introduce  by Drezner and Farnum \cite[Section 3]{DF93}. 
Under the current setting, 
$\{X_n'\}_{n=1}^{\infty}$ forms the dependent Bernoulli sequence 
studied by \cite{JJQ08, WQY12}. 
In particular, if we let $S_0'=0$ and $S_n'=X_1'+\dots+X_n' \ (n\geq 1)$, 
then $\{S_n\}_{n=0}^{\infty}$ corresponds to 
the dependent Bernoulli process $\{S_n'\}_{n=0}^{\infty}$ by the relation $S_n=2S_n'-n$.

Let $a_1=1$ and $b_1=1$, and let  
$$a_n=\prod_{k=1}^{n-1}\left(1+\frac{\alpha_k}{k}\right), 
\quad b_n=\prod_{k=1}^{n-1}\left(1+\frac{2\alpha_k}{k}\right) \quad (n\geq 2).$$
For $n\geq 1$, we define ${\cal F}_n=\sigma(X_1,\dots, X_n)$ 
and 
$M_n=(S_n-E[S_n])/a_n$.
The next lemma can be proved in a way similar to \cite[Section 3]{KT19}.
\begin{lem}\label{lem:moment}
\begin{enumerate}
\item[{\rm (1)}]
For any $n\geq 1$,
$$
E[X_{n+1}\mid {\cal F}_n]=\frac{\alpha_n}{n}S_n+(1-\alpha_n)\varepsilon_n
$$
and 
$$
E[S_{n+1}\mid {\cal F}_n]=\left(1+\frac{\alpha_n}{n}\right)S_n+(1-\alpha_n)\varepsilon_n.
$$
Moreover, for any $n\geq 1$,
$$E[S_n]=a_n\left(2q-1+\sum_{k=1}^{n-1}\frac{(1-\alpha_k)\varepsilon_k}{a_k}\right).$$

\item[{\rm (2)}] 
For any $n\geq 1$, 
$$E[S_{n+1}^2\mid {\cal F}_n]=\left(1+\frac{2\alpha_n}{n}\right) S_n^2+2(1-\alpha_n)\varepsilon_nS_n+1$$
and 
$$E[S_n^2]=b_n\left(\sum_{k=1}^{n}\frac{1}{b_k}+2\sum_{k=1}^{n-1}\frac{(1-\alpha_k)\varepsilon_k}{b_{k+1}}E[S_k]\right).$$

\item[{\rm (3)}] $\{M_n\}_{n=1}^{\infty}$ is a square integrable martingale with respect to the filtration $\{{\cal F}_n\}_{n\geq 1}$.
\end{enumerate}
\end{lem}

We next present asymptotic properties of series and products  
which will be used in the subsequent sections. 
Let $g_1=1$ and $l_1=1$, and let 
\begin{equation}\label{eq:slow-seq}
g_n=\prod_{k=1}^{n-1}\left(1+\frac{\alpha}{k}\right), \quad 
l_n=\prod_{k=1}^{n-1}\left(1+\frac{\alpha_k-\alpha}{k+\alpha}\right) \quad (n\geq 2).
\end{equation}
Then by definition, $a_n=g_nl_n$. 
Let  $\rho_1=1$ and 
$$\rho_n=\exp\left(\sum_{k=1}^{n-1}\frac{\alpha_k-\alpha}{k}\right) \quad (n\geq 2).$$
\begin{lem}\label{lem:series}
\begin{enumerate}
\item[{\rm (1)}] For any $\alpha\in [0,1]$,
\begin{equation}\label{eq:prod-asymp}
n^{1-\alpha}\left(g_n-\frac{n^{\alpha}}{\Gamma(\alpha+1)}\right)
\rightarrow -\frac{\alpha(1-\alpha)}{2\Gamma(\alpha+1)}\quad (n\rightarrow\infty).
\end{equation}
\item[{\rm (2)}] 
There exists a positive constant $C_0$ such that $b_n\sim C_0a_n^2$.

\item[{\rm (3)}]
 $a_n/n\rightarrow 0 \ (n\rightarrow\infty)$ if and only if 
\begin{equation}\label{eq:conv-alpha}
\sum_{n=1}^{\infty}\frac{1-\alpha_n}{n}=\infty.
\end{equation}
\item[{\rm (4)}] 
Assume that the sequence $\{\alpha_n\}_{n=1}^{\infty}$ is convergent  to some $\alpha\in [0,1]$. 
Then the sequence $\{l_n\}_{n=1}^{\infty}$ is slowly varying, that is, 
for any $\lambda>0$,
$$\lim_{n\rightarrow\infty}\frac{l_{[\lambda n]}}{l_n}=1.$$
Moreover, there exists a positive constant $C_*$ such that $l_n\sim C_*\rho_n$.
In particular, 
$$a_n\sim \frac{C_*}{\Gamma(\alpha+1)}n^{\alpha}\rho_n$$
and $\{a_n\}_{n=1}^{\infty}$ is a regularly varying sequence with index $\alpha$. 

\item[{\rm (5)}] If $\alpha=1$, then \eqref{eq:conv-alpha} holds if and only if  
\begin{equation}\label{eq:conv-a}
\sum_{n=1}^{\infty}\frac{1-\alpha_n}{a_n}=\infty.
\end{equation}
Moreover, if this condition is valid, then 
\begin{equation}\label{eq:sum-a}
\lim_{n\rightarrow\infty}\rho_n\sum_{k=1}^{n}\frac{1-\alpha_k}{a_k}=\frac{1}{C_*}
\end{equation}
and 
\begin{equation}\label{eq:sum-b}
\lim_{n\rightarrow\infty}\frac{a_n}{n}\sum_{k=1}^n\frac{1-\alpha_k}{a_k}=1.
\end{equation}
\end{enumerate}
\end{lem}

\begin{proof}
(3) is proved in \cite[Lemma 2]{WQY12}.  
We prove (1), (2), (4) and (5) in this order. 

By the Stirling formula for the Gamma function 
(see, e.g., \cite[Table 1]{W68}), 
$$\Gamma(t)=\sqrt{\frac{2\pi}{t}}\left(\frac{t}{e}\right)^{t}
\left(1+\frac{1}{12t}+O\left(\frac{1}{t^2}\right)\right) \quad (t\rightarrow\infty).$$
Hence we have as $n\rightarrow\infty$,
$$g_n=\frac{\Gamma(n+\alpha)}{\Gamma(n)\Gamma(\alpha+1)}
=\frac{1}{\Gamma(\alpha+1)e^{\alpha}}\sqrt{\frac{n}{n+\alpha}}\left(1+\frac{\alpha}{n}\right)^n(n+\alpha)^{\alpha}
\left(1+O\left(\frac{1}{n^2}\right)\right).$$ 
Then by elementary calculus, we obtain (1).

Since
$$
\frac{b_n}{a_n}=\prod_{k=1}^{n-1}\frac{k+2\alpha_k}{k+\alpha_k}
=\prod_{k=1}^{n-1}\left(1+\frac{\alpha_k}{k+\alpha_k}\right)
=a_n\prod_{k=1}^{n-1}\left(1-\frac{\alpha_k^2}{k^2+2k\alpha_k+\alpha_k^2}\right)
$$
and 
$$0\leq \sum_{n=1}^{\infty}\frac{\alpha_n^2}{n^2+2n\alpha_n+\alpha_n^2}
\leq \sum_{n=1}^{\infty}\frac{1}{n^2}<\infty,$$
the infinite product 
$$C_0:=\prod_{n=1}^{\infty}\left(1-\frac{\alpha_n^2}{n^2+2n\alpha_n+\alpha_n^2}\right)$$
is convergent so that we have (2).

By the Taylor theorem, for any $x>-1$, there exists a constant $\theta\in (0,1)$ such that 
$$\log(1+x)=x-\frac{(\theta x)^2}{2}.$$
Hence for all sufficiently large $k\geq 1$, 
there exist constants $\theta_k\in (0,1)$ such that
$$
\log\left(1+\frac{\alpha_k-\alpha}{k+\alpha}\right)
=\frac{\alpha_k-\alpha}{k+\alpha}-\frac{\theta_k^2}{2}\left(\frac{\alpha_k-\alpha}{k+\alpha}\right)^2
=\frac{\alpha_k-\alpha}{k}
-\left(\frac{\alpha(\alpha_k-\alpha)}{k(k+\alpha)}
+\frac{\theta_k^2}{2}\left(\frac{\alpha_k-\alpha}{k+\alpha}\right)^2\right).
$$
Combining this relation with the expression 
$$l_n=\exp\left\{\sum_{k=1}^{n-1}\log\left(1+\frac{\alpha_k-\alpha}{k+\alpha}\right)\right\},$$
we have 
$\l_n/\rho_n\rightarrow C_* \ (n\rightarrow\infty)$ with 
$$C_*:=\exp\left\{-\sum_{n=1}^{\infty}
\left(\frac{\alpha(\alpha_n-\alpha)}{n(n+\alpha)}
+\frac{\theta_n^2}{2}\left(\frac{\alpha_n-\alpha}{n+\alpha}\right)^2\right)\right\}.$$ 
Since the sequence $\{\rho_n\}_{n=1}^{\infty}$ is slowly varying, so is the sequence $\{l_n\}_{n=1}^{\infty}$.
Then by (1), the proof of (4) is complete. 

Suppose that $\alpha=1$. 
If \eqref{eq:conv-alpha} holds, 
then \eqref{eq:conv-a} follows by (3). 
We now assume that \eqref{eq:conv-a} holds but \eqref{eq:conv-alpha} fails. 
Then by (4),
\begin{equation}\label{eq:seq-sum}
\begin{split}
\sum_{k=1}^{n-1}\frac{1-\alpha_k}{a_k}\sim \frac{1}{C_*}\sum_{k=1}^n\frac{1-\alpha_k}{k \rho_{k+1}}
=\frac{1}{C_*}\sum_{k=1}^n\frac{1-\alpha_k}{k}\exp\left(\sum_{l=1}^k\frac{1-\alpha_l}{l}\right).
\end{split}
\end{equation} 
Since the right hand side above is convergent as $n\rightarrow\infty$ by assumption, 
we have a contradiction so that \eqref{eq:conv-alpha} holds. 
Then \eqref{eq:series-order-2} and \eqref{eq:seq-sum} yield \eqref{eq:sum-a}. 
Combining this with  (4), we get \eqref{eq:sum-b}.  
\end{proof}

\section{Law of the large numbers}\label{sect:lln}
In this section, we establish  the law of the large numbers 
for $\{S_n\}_{n=0}^{\infty}$.

\begin{thm}\label{thm:lln-0}
Let $\{r_n\}_{n=1}^{\infty}$ be a positive sequence such that $a_n/r_n\rightarrow 0$ as $n\rightarrow\infty$ 
and $\sum_{n=1}^{\infty}1/r_n^2<\infty$.
Then 
$$\lim_{n\rightarrow\infty}\frac{S_n-E[S_n]}{r_n}=0, \quad \text{$P$-a.s.}$$
\end{thm}

If $r_n=n \ (n\geq 1)$, then this theorem is proved by 
\cite[Theorem 1]{CGS17}, \cite[Theorem 2.1]{JJQ08}, 
\cite[Theorem 1]{KT19} and \cite[Theorem 1]{WQY12}.  
As we will see from the proof of Theorem \ref{thm:lln-0} below, 
their approach still works for more general sequences $\{r_n\}_{n=1}^{\infty}$ 
because the sequence $\{(X_n-E[X_n\mid {\cal F}_{n-1}])/r_n\}_{n=1}^{\infty}$ is a martingale difference.
\begin{proof}[Proof of Theorem {\rm \ref{thm:lln-0}}]
Let 
$$
d_j=\begin{cases}M_1 & (j=1),\\ 
M_{j}-M_{j-1} & (j\geq 2)
\end{cases}
$$
and $\gamma_j=1+\alpha_j/j$.
Then by Lemma \ref{lem:moment} (1),
\begin{equation}\label{eq:difference}
\begin{split}
d_j
&=\frac{S_j-E[S_j]-\gamma_{j-1}(S_{j-1}-E[S_{j-1}])}{a_j}
=\frac{S_j-\gamma_{j-1}S_{j-1}-(1-\alpha_{j-1})\varepsilon_{j-1}}{a_j}\\
&=\frac{S_j-E[S_j\mid {\cal F}_{j-1}]}{a_j}
=\frac{X_j-E[X_j\mid {\cal F}_{j-1}]}{a_j},
\end{split}
\end{equation}
which yields
\begin{equation}\label{eq:cond-d}
\frac{X_j-E[X_j\mid {\cal F}_{j-1}]}{r_j}=\frac{d_ja_j}{r_j}.
\end{equation}

By assumption, 
$$\sum_{j=1}^{\infty}E\left[\left(\frac{X_j-E[X_j\mid {\cal F}_{j-1}]}{r_j}\right)^2\mid{\cal F}_{j-1}\right]
\leq \sum_{j=1}^{\infty}\frac{4}{r_j^2}<\infty.$$
Hence by \cite[Theorem 2.15]{HH80} and \eqref{eq:cond-d}, the series
$$\sum_{j=1}^{\infty}\frac{X_j-E[X_j\mid{\cal F}_{j-1}]}{r_j}
=\sum_{j=1}^{\infty}\frac{d_j}{r_j/a_j}$$
converges almost surely.
Since $a_n/r_n\rightarrow 0 \ (n\rightarrow \infty)$ by assumption, 
we obtain by Kronecker's lemma,
$$\frac{S_n-E[S_n]}{r_n}=\frac{a_n}{r_n}M_n=\frac{a_n}{r_n}\sum_{j=1}^nd_j\rightarrow 0, \quad \text{$P$-a.s.}$$
The proof is complete.
\end{proof}

By Theorem \ref{thm:lln-0} with Lemma \ref{lem:moment} (1), we have the growth exponent of $\{S_n\}_{n=0}^{\infty}$:
\begin{cor}\label{cor:lln-0}
Let 
$$r_n=a_n\sum_{k=1}^n\frac{(1-\alpha_k)\varepsilon_k}{a_k}.$$
If $a_n/r_n\rightarrow 0 \ (n\rightarrow\infty)$ and $\sum_{n=1}^{\infty}1/r_n^2<\infty$, then 
\begin{equation}\label{eq:lln-0}
\lim_{n\rightarrow\infty}\frac{S_n}{r_n}=1, \quad \text{$P$-a.s.}
\end{equation}
\end{cor}

We can also prove the $L^2$-convergence of $S_n/r_n$ as in \cite[Theorem 4]{KT19}.
Let $w_n=\sum_{k=1}^n 1/{a_k^2}$. 
\begin{thm}\label{thm:conv-limit}
Let 
$$r_n=a_n\sum_{k=1}^n\frac{(1-\alpha_k)\varepsilon_k}{a_k}.$$
If $a_n/r_n\rightarrow 0 \ (n\rightarrow\infty)$, then $E[S_n]\sim r_n$.
Moreover, if $a_n\sqrt{w_n}/r_n\rightarrow 0 \ (n\rightarrow\infty)$,
then $E[S_n^2]\sim E[S_n]^2$ and 
\begin{equation}\label{eq:div-3}
\lim_{n\rightarrow\infty}\frac{S_n}{r_n}=1 \quad \text {in $L^2(P)$.}
\end{equation}
\end{thm}

\begin{proof}
Assume that $a_n/r_n\rightarrow 0 \ (n\rightarrow\infty)$. 
Then by Lemma \ref{lem:moment} (1), $E[S_n]\sim r_n$.
Hence by Lemma \ref{lem:series} (2) and Lemma \ref{lem:series-order},
\begin{equation*}
\begin{split}
2b_n\sum_{k=1}^{n-1}\frac{(1-\alpha_k)\varepsilon_k}{b_{k+1}}E[S_k]
&\sim 2 a_n^2\sum_{k=1}^n\frac{(1-\alpha_k)\varepsilon_k}{a_k}\left(\sum_{l=1}^k\frac{(1-\alpha_l)\varepsilon_l}{a_l}\right)\\
&\sim a_n^2
\left(\sum_{k=1}^n\frac{(1-\alpha_k)\varepsilon_k}{a_k}\right)^2
\sim E[S_n]^2.
\end{split}
\end{equation*}
Lemma \ref{lem:series} (2) also implies that for some positive constants $c_1$ and $c_2$, 
$c_1 \sum_{k=1}^{n}1/b_k\leq w_n\leq c_2 \sum_{k=1}^{n}1/b_k \ (n\geq 1)$. 
Therefore, if $a_n\sqrt{w_n}/r_n\rightarrow 0 \ (n\rightarrow\infty)$, then
Lemma \ref{lem:moment} (2) yields $E[S_n^2]\sim E[S_n]^2$.  
Since
$$\lim_{n\rightarrow\infty}\frac{E[(S_n-E[S_n])^2]}{E[S_n]^2}=0,$$
we obtain \eqref{eq:div-3}.
\end{proof}

\begin{rem}\label{rem:as-conv}\rm 
Let $\alpha_n=\alpha\in [0,1)$ and $\varepsilon_n=n^{-\rho}$ for some $\rho>0$. 
Then by Corollary \ref{cor:lln-0} and Theorem \ref{thm:conv-limit}, we have:
\begin{itemize}
\item If $0\leq \alpha<1-\rho$ and $0<\rho<1/2$, then 
$$\lim_{n\rightarrow\infty}\frac{S_n}{n^{1-\rho}}=\frac{1-\alpha}{1-(\alpha+\rho)}, \ \text{$P$-a.s.\ and in $L^2(P)$.}$$
\item If $\alpha=1-\rho$ and $0<\rho<1/2$, then 
$$\lim_{n\rightarrow\infty}\frac{S_n}{n^{\alpha}\log n}=1-\alpha, \ \text{$P$-a.s.\ and in $L^2(P)$.}$$
\item If $\alpha=\rho=1/2$, then 
$$\lim_{n\rightarrow\infty}\frac{S_n}{\sqrt{n}\log n}=\frac{1}{2}, \ \text{$P$-a.s.\ and in $L^2(P)$}$$
\end{itemize}
(see the subsequent examples for the validity of these calculations).
The $L^2$-convergence results above are already proved in \cite[Theorem 4]{KT19}.
\end{rem}

We now apply Theorem \ref{thm:lln-0} to find the scaling factor  of $\{S_n\}_{n=0}^{\infty}$. 
We first see that  if $\varepsilon>0$, 
then $\{S_n\}_{n=0}^{\infty}$ grows linearly under reasonable conditions.
\begin{exam}\label{exam:lln}\rm 
Let $0\leq \varepsilon\leq 1$. 

\noindent
(i)  Assume that $0\leq \alpha<1$. 
Since $\{a_n\}_{n=1}^{\infty}$ is a regularly varying sequence with index $\alpha$, we have 
$$\lim_{n\rightarrow\infty}\frac{a_n}{n}\sum_{k=1}^n\frac{1}{a_k}=\frac{1}{1-\alpha}$$
and thus 
\begin{equation}\label{eq:lim-linear}
\lim_{n\rightarrow\infty}\frac{r_n}{n}=\lim_{n\rightarrow\infty}\frac{a_n}{n}\sum_{k=1}^n\frac{(1-\alpha_k)\varepsilon_k}{a_k}=\varepsilon.
\end{equation}
As $a_n/n\rightarrow 0 \ (n\rightarrow\infty)$,  
we see that if $\varepsilon>0$, then by Corollary \ref{cor:lln-0} and Theorem \ref{thm:conv-limit}, 
\begin{equation}\label{eq:aslim-linear}
\lim_{n\rightarrow\infty}\frac{S_n}{n}=\varepsilon, \quad \text{$P$-a.s.\ and in $L^2(P)$.}
\end{equation}
This equality is valid also for $\varepsilon=0$ by Theorem \ref{thm:lln-0} and the calculation of $E[(S_n/n)^2]$. 
Moreover, since it follows by \eqref{eq:dist} that
$$P(X_n=\pm 1 \mid {\cal F}_{n-1})
=\frac{\alpha_n}{2}\left(1\pm \frac{S_n}{n}\right)+(1-\alpha_n)\frac{1\pm\varepsilon_n}{2},$$
we have 
\begin{equation}\label{eq:bernoulli}
P(X_n=\pm 1 \mid {\cal F}_{n-1})\rightarrow \frac{1\pm\varepsilon}{2}, \quad \text{$P$-a.s.\ and in $L^2(P)$.}
\end{equation}

\noindent
(ii) Assume that $\alpha=1$ and $\sum_{n=1}^{\infty}(1-\alpha_n)/n=\infty$. 
Since \eqref{eq:lim-linear} remains valid by \eqref{eq:sum-b}, 
we get \eqref{eq:aslim-linear}  and \eqref{eq:bernoulli}. 
\end{exam}

We next see that if $\varepsilon=0$ and $0\leq \alpha<1$, 
then the convergence rate of $\{\varepsilon_n\}_{n=1}^{\infty}$ affects 
the scaling factor of $\{S_n\}_{n=0}^{\infty}$. 

\begin{exam}\label{exam:lln-1}\rm 
Let $\varepsilon=0$ and $0\leq \alpha<1$. 
Assume that for some $\rho\in [0,1-\alpha)$, 
the sequence $\{\varepsilon_n\}_{n=1}^{\infty}$ is regularly varying 
with index $-\rho$. 
Let $r_n$ be as in Corollary \ref{cor:lln-0}. 
Since $\alpha+\rho<1$, we have 
$$\frac{r_n}{a_n}=\sum_{k=1}^n \frac{(1-\alpha_k)\varepsilon_k}{a_k}
\sim \frac{1-\alpha}{1-(\alpha+\rho)}\frac{n\varepsilon_n}{a_n}.$$
Namely,  $a_n/r_n\rightarrow 0$ and 
$$r_n\sim \frac{1-\alpha}{1-(\alpha+\rho)}n\varepsilon_n.$$

We now impose  the following condition on $\alpha$, $\rho$ and $\{\varepsilon_n\}_{n=1}^{\infty}$:
\begin{itemize}
\item $0\leq \rho<1/2\wedge(1-\alpha)$, or 
\item $\rho=1/2$ and for some positive constants $\eta$ and $c_1$, 
\begin{equation}\label{eq:as-bound}
\sqrt{n}\varepsilon_n\geq c_1(\log n)^{(1+\eta)/2} \ (n\geq 1).
\end{equation}
\end{itemize}
Note that under this condition, $0\leq \alpha<1/2$.
Since $\sum_{n=1}^{\infty}1/r_n^2<\infty$, 
we have by Corollary \ref{cor:lln-0}, 
\begin{equation}\label{eq:a-c-1}
\lim_{n\rightarrow\infty}\frac{S_n}{n\varepsilon_n}=\frac{1-\alpha}{1-(\alpha+\rho)}, \quad \text{$P$-a.s.}
\end{equation}
If we replace the condition \eqref{eq:as-bound} with $\lim_{n\rightarrow\infty}\sqrt{n}\varepsilon_n=\infty$,
then by Theorem \ref{thm:conv-limit}, \eqref{eq:a-c-1} holds also in $L^2(P)$. 
\end{exam}

\begin{exam}\label{exam:0-case}\rm 
Let $\varepsilon=0$ and $0\leq \alpha<1$. 
Assume that 
the sequence $\{\varepsilon_n\}_{n=1}^{\infty}$ is regularly varying  
with index $-(1-\alpha)$. 
Let $r_n$ be as in Corollary \ref{cor:lln-0}. 
Then the scaling factor of $\{S_n\}_{n=0}^{\infty}$ depends also  
on the asymptotic behavior of $\{\alpha_n\}_{n=1}^{\infty}$. 
In what follows, we assume that 
for some constants $\eta$, $\kappa$ and $\theta>0$, 
$$\varepsilon_n=\frac{(\log n)^{\eta-1}}{n^{1-\alpha}}, 
\quad \alpha_n=\alpha+\frac{\kappa}{(\log n)^{\theta}} \quad (n\geq 2).$$

\noindent
(i) Let $0<\theta<1$. 
If $\kappa<0$, then by Example \ref{exam:order-l-a-1} with Lemma \ref{lem:series} (4), 
$$r_n\sim \frac{1-\alpha}{-\kappa}n^{\alpha}(\log n)^{\eta+\theta-1}$$
and $a_n/r_n\rightarrow 0 \ (n\rightarrow\infty)$. 
In particular, if 
\begin{itemize}
\item $0<\theta<1$, $\alpha>1/2$ and $\kappa<0$, or 
\item $0<\theta<1$, $\alpha=1/2$, $\kappa<0$ and $\eta+\theta>3/2$, 
\end{itemize}
then $\sum_{n=1}^{\infty}1/r_n^2<\infty$ so that $S_n/r_n\rightarrow 1$ $P$-a.s.\ by Corollary \ref{cor:lln-0}. 
If we replace the condition $\eta+\theta>3/2$ with $\eta+\theta/2>1$
and keep other conditions, 
then $a_n\sqrt{w_n}/r_n\rightarrow 0 \ (n\rightarrow\infty)$ 
by Examples \ref{exam:order-l-a-1} and \ref{exam:order-l-a-2}.  
Hence $S_n/r_n\rightarrow 1$ in $L^2(P)$ by Theorem \ref{thm:conv-limit}.

\noindent
(ii) Let $\theta=1$. If $\eta\geq \kappa$, then 
by Example \ref{exam:order-l-a-1} with Lemma \ref{lem:series} (4), 
$$r_n\sim \begin{cases}
\displaystyle \frac{1-\alpha}{\eta-\kappa}n^{\alpha}(\log n)^{\eta} & (\theta=1, \eta>\kappa),\\
(1-\alpha)n^{\alpha}(\log n)^{\eta}\log\log n & (\theta=1, \eta=\kappa)
\end{cases}$$
and $a_n/r_n\rightarrow 0 \ (n\rightarrow\infty)$. 
In particular, if 
\begin{itemize}
\item $\theta=1$, $\alpha>1/2$ and $\eta\geq \kappa$, or 
\item $\theta=1$, $\alpha=1/2$, and $\eta=\kappa\geq 1/2$ or $\eta>\kappa\vee (1/2)$, 
\end{itemize}
then by Examples \ref{exam:order-l-a-1} and \ref{exam:order-l-a-2}, 
$\sum_{n=1}^{\infty}1/r_n^2<\infty$ and $a_n\sqrt{w_n}/r_n\rightarrow 0 \ (n\rightarrow\infty)$ 
so that $S_n/r_n\rightarrow 1$ $P$-a.s.\ and in $L^2(P)$ 
by Corollary \ref{cor:lln-0} and Theorem \ref{thm:conv-limit}, respectively.

\noindent
(iii) Let $\theta>1$ or $\kappa=0$. 
If $\eta\geq 0$, then all the calculations for $\theta=1$ remain valid by taking $\kappa=0$. 
Therefore, if 
\begin{itemize}
\item $\theta>1$, $\alpha>1/2$ and $\eta\geq 0$, or 
\item $\theta>1$, $\alpha=1/2$ and $\eta>1/2$, 
\end{itemize}
then  $S_n/r_n\rightarrow 1$ $P$-a.s.\ and in $L^2(P)$ 
by Corollary \ref{cor:lln-0} and Theorem \ref{thm:conv-limit}, respectively.
\end{exam}

\begin{exam}\label{exam:lln-2}\rm 
Assume that for some positive constants $\eta$, $\kappa$ and $\theta$, 
$$\varepsilon_n=\frac{1}{(\log n)^{\eta}}, \quad \alpha_n=1-\frac{\kappa}{(\log n)^{\theta}} \quad (n\geq 1).$$
By Example \ref{exam:order-l-a-1} with Lemma \ref{lem:series} (4), 
if $0<\theta<1$, or if $\theta=1$ and $\kappa\geq \eta$, then 
$$r_n\sim
\begin{cases}
\displaystyle \frac{n}{(\log n)^{\eta}} & (0<\theta<1),\\
\displaystyle \frac{\kappa}{\kappa-\eta}\frac{n}{(\log n)^{\eta}} & (\theta=1, \kappa>\eta),\\
\displaystyle \frac{\kappa n\log\log n}{(\log n)^{\kappa}} & (\theta=1,\kappa=\eta)
\end{cases}$$
and $a_n/r_n\rightarrow 0 \ (n\rightarrow\infty)$.  
Since $\sum_{n=1}^{\infty}1/r_n^2<\infty$ and $\sum_{n=1}^{\infty}1/a_n^2<\infty$, 
we have $S_n/r_n\rightarrow 1$ $P$-a.s.\ and in $L^2(P)$ 
by Corollary \ref{cor:lln-0} and Theorem \ref{thm:conv-limit}, respectively.
\end{exam}

\section{Gaussian fluctuations}\label{sect:gauss}
In this section, we first reveal the Gaussian fluctuation of $\{S_n\}_{n=0}^{\infty}$ around the mean 
by proving the central limit theorem and law of the iterated logarithm. 
We next show that even if such a fluctuation is unavailable, 
the fluctuation around the random drift $a_nM_{\infty}$ is Gaussian. 
The former is proved in \cite{CGS17,JJQ08,KT19,WQY12} 
for the dependent Bernoulli processes and superdiffusive elephant random walks with $\varepsilon<1$; 
the latter is proved in  \cite{KT19} for the superdiffusive elephant random walks with $\varepsilon<1$.   
Our results in this section extend the both results above
to the random walks under consideration which can be degenerate in the sense that $\varepsilon=1$.

\subsection{Central limit theorem and law of the iterated logarithm}

Let $w_n=\sum_{k=1}^{n}1/a_k^2$ and $\phi(t)=\sqrt{2t\log\log t}$. 
We state the results separately in Theorems \ref{thm:fluct-0} and \ref{thm:quad-var} below. 

\begin{thm}\label{thm:fluct-0}
\begin{enumerate}
\item[{\rm (1)}] 
Let $\varepsilon\in [0,1)$. If $\sum_{n=1}^{\infty}1/a_n^2=\infty$, then 
$$
\sum_{k=1}^nE[d_k^2\mid {\cal F}_{k-1}]
\sim (1-\varepsilon^2)w_n.
$$
Moreover, 
$$\frac{S_n-E[S_n]}{a_n\sqrt{w_n}}\rightarrow N(0,1-\varepsilon^2) \quad \text{in distribution}$$
and 
$$\limsup_{n\rightarrow\infty}\pm \frac{S_n-E[S_n]}{a_n \phi(w_n)}=\sqrt{1-\varepsilon^2}, \quad \text{$P$-a.s.}$$
\item[{\rm (2)}] 
Let $\varepsilon\in [0,1]$. 
If $\sum_{n=1}^{\infty}1/a_n^2<\infty$, 
then 
\begin{equation}\label{eq:lim-m}
\lim_{n\rightarrow\infty}\frac{S_n-E[S_n]}{a_n}=M_{\infty}, \quad \text{$P$-a.s.\ and in $L^2(P)$}.
\end{equation}
In particular, $E[M_{\infty}]=0$ and $P(M_{\infty}\ne 0)>0$.
\end{enumerate}
\end{thm}

We omit the proof of Theorem \ref{thm:fluct-0} 
because this theorem follows by \cite[Theorems 3 and 4]{WQY12}.

\begin{rem}\label{rem:wn}\rm
Theorem \ref{thm:fluct-0}  says that if $\varepsilon\in [0,1)$ 
and $\alpha\ne 1/2$,
then the fluctuations of 
$\{\varepsilon_n\}_{n=1}^{\infty}$ and $\{\alpha_n\}_{n=1}^{\infty}$ 
do not essentially affect 
that of $\{S_n\}_{n=0}^{\infty}$ around the mean.
In fact, we have:
\begin{itemize}
\item If $0\leq \alpha<1/2$, then 
$a_n\sqrt{w_n}\sim \sqrt{n/(1-2\alpha)}$ 
so that $\{S_n\}_{n=1}^{\infty}$ exhibits the diffusive behavior. 
\item If $1/2<\alpha<1$, 
then $\{S_n\}_{n=1}^{\infty}$ exhibits the superdiffusive behavior 
because $a_n\sim n^{\alpha}l_n/\Gamma(\alpha+1)$ 
with the slowly varying sequence $\{l_n\}_{n=1}^{\infty}$ in \eqref{eq:slow-seq}.
\end{itemize}
On the other hand, if $\alpha=1/2$, then the behavior of $\{S_n\}_{n=1}^{\infty}$ depends on 
the convergence rate of $\{\alpha_n\}_{n=1}^{\infty}$ 
because this rate affects whether $\sum_{n=1}^{\infty}1/a_n^2$ is convergent or not 
(see Example \ref{exam:clt-1} below for details). 
This fact is already observed  in \cite{JJQ08} for  
the constant bias $\varepsilon_n=\varepsilon \ (n\geq 1)$, 
and in \cite{KT19} for the constant correlation $\alpha_n=\alpha \ (n\geq 1)$.
\end{rem}

We are now concerned with the fluctuation of $\{S_n\}_{n=1}^{\infty}$ around the mean  
for $\varepsilon=1$ and $\sum_{n=1}^{\infty}1/a_n^2=\infty$, 
which is not covered by Theorem \ref{thm:fluct-0}. 
Let us make the next assumptions on the sequences $\{\varepsilon_n\}_{n=1}^{\infty}$ and $\{\alpha_n\}_{n=1}^{\infty}$.
\begin{assum}\label{assum:alpha} \rm 
The sequences $\{\varepsilon_n\}_{n=1}^{\infty}$ and $\{\alpha_n\}_{n=1}^{\infty}$ 
satisfy the next conditions:
\begin{enumerate}
\item[(i)] $\varepsilon=1$ and $\sum_{n=1}^{\infty}1/a_n^2=\infty$.
\item[(ii)] The sequence $\{1-\varepsilon_n\}_{n=1}^{\infty}$ is regularly varying 
with index $-\rho$ for some $\rho\in [0,1/2)$.
\end{enumerate}
\end{assum}

Let $v_n=\sum_{k=1}^n(1-\varepsilon_k^2)/a_k^2$ and 
$$c_{\alpha,\rho}=\frac{(1-\alpha)(1-\rho)}{1-(\alpha+\rho)}.$$

\begin{thm}\label{thm:quad-var} 
If Assumption {\rm \ref{assum:alpha}} is fulfilled, then the next assertions hold{\rm :}
\begin{enumerate}
\item[{\rm (1)}] 
If $\sum_{n=1}^{\infty}(1-\varepsilon_n)/a_n^2=\infty$, then
\begin{equation}\label{eq:quad-asymp}
\sum_{k=1}^nE[d_k^2\mid {\cal F}_{k-1}]
\sim c_{\alpha,\rho}v_n.
\end{equation}
Moreover, 
$$\frac{S_n-E[S_n]}{a_n\sqrt{v_n}}\rightarrow N(0,c_{\alpha,\rho}) \quad \text{in distribution}$$
and 
$$\limsup_{n\rightarrow\infty}\pm \frac{S_n-E[S_n]}{a_n\phi(v_n)}=\sqrt{c_{\alpha,\rho}}, \quad \text{$P$-a.s.}$$
\item[{\rm (2)}] 
If $\sum_{n=1}^{\infty}(1-\varepsilon_n)/a_n^2<\infty$, 
then the same assertion as in Theorem {\rm \ref{thm:fluct-0} (2)} holds. 
\end{enumerate}
\end{thm}

Theorem \ref{thm:quad-var}  reveals how  the convergence rate of the bias $\{\varepsilon_n\}_{n=1}^{\infty}$ affects 
the fluctuation of  $\{S_n\}_{n=0}^{\infty}$ around the mean. 
In particular, we find that $\rho=1-2\alpha$ is the borderline between (1) and (2).

We also note that if Assumption {\rm \ref{assum:alpha}} is fulfilled,  
then \eqref{eq:aslim-linear} and \eqref{eq:bernoulli} hold by Example \ref{exam:lln}. 
Therefore, if we assume in addition that  $\sum_{n=1}^{\infty}(1-\varepsilon_n)/a_n^2=\infty$, then 
$$v_n\sim \sum_{k=1}^n\frac{4}{a_k^2}P(X_k=1)(1-P(X_k=1))$$
and 
$$\lim_{n\rightarrow\infty}\frac{1}{v_n}\sum_{k=1}^{n-1}\frac{1}{a_k^2}= \infty.$$
Namely, we can not apply \cite[Theorems 3 and 4]{WQY12} to the proof of Theorem \ref{thm:quad-var}.

To prove Theorem \ref{thm:quad-var}, 
we find the convergence rate for Example \ref{exam:lln} with $\varepsilon=1$.
\begin{lem}\label{lem:lln-order}
Let Assumption {\rm \ref{assum:alpha}} hold. 
\begin{enumerate}
\item[{\rm (1)}] There exists a positive constant $c_0$ such that 
for any $n\geq 2$, 
$$
\left|(1-\alpha)\frac{a_{n-1}}{n}\sum_{k=1}^{n-1}\frac{1}{a_k}-1
-\frac{a_{n-1}}{n}\sum_{k=1}^{n-1}\frac{\alpha_k-\alpha}{a_k}\right|
\leq \frac{c_0a_{n-1}}{n}.
$$

\item[{\rm (2)}] 
The next assertion holds{\rm :}
$$\lim_{n\rightarrow\infty}\frac{1}{1-\varepsilon_n}\left(\frac{S_n}{n}-1\right)=-\frac{1-\alpha}{1-(\alpha+\rho)},  \ \text{$P$-a.s.}$$ 
\end{enumerate}
\end{lem}
\begin{proof}
We first prove (1). Let 
\begin{equation}\label{eq:sum-small}
\begin{split}
&(1-\alpha)\frac{a_{n-1}}{n}\sum_{k=1}^{n-1}\frac{1}{a_k}-1-\frac{a_{n-1}}{n}\sum_{k=1}^{n-1}\frac{\alpha_k-\alpha}{a_k}
= (1-\alpha)\frac{a_{n-1}}{n}
\left\{\sum_{k=1}^{n-1}\left(\frac{1}{g_k}-\frac{\Gamma(\alpha+1)}{k^{\alpha}}\right)\frac{1}{l_k}\right\}\\
&+\frac{a_{n-1}}{n}\left\{\Gamma(\alpha+1)\left((1-\alpha)\sum_{k=1}^{n-1}\frac{1}{k^{\alpha} l_k}-\frac{n^{1-\alpha}}{l_n}\right)
-\sum_{k=1}^{n-1}\frac{\alpha_k-\alpha}{a_k}\right\}
+\left(\Gamma(\alpha+1)\frac{a_{n-1}}{n^{\alpha}l_n}-1\right)\\
&={\rm (I)}+{\rm (II)}+{\rm (III)}.
\end{split}
\end{equation}
Since Lemma \ref{lem:series} (1) yields 
\begin{equation}\label{eq:asymp-inv}
\frac{1}{g_n}-\frac{\Gamma(\alpha+1)}{n^{\alpha}}
=-\frac{\Gamma(\alpha+1)}{n^{\alpha}g_n}\left(g_n-\frac{n^{\alpha}}{\Gamma(\alpha+1)}\right)
\sim \frac{c_1}{n^{1+\alpha}}
\end{equation}
for some $c_1>0$, we have
$$
\left|\sum_{k=1}^{n-1}\left(\frac{1}{g_k}-\frac{\Gamma(\alpha+1)}{k^{\alpha}}\right)\frac{1}{l_k}\right|
\leq \sum_{k=1}^{n-1}\left|\frac{1}{g_k}-\frac{\Gamma(\alpha+1)}{k^{\alpha}}\right|\frac{1}{l_k}
\leq c_2\sum_{k=1}^{\infty}\frac{1}{k^{1+\alpha}l_k}<\infty
$$
so that 
$$|{\rm (I)}|\leq c_3\frac{a_n}{n}.$$

For $k\geq 2$, 
since
$$
\frac{1-\alpha}{k^{\alpha}}\leq k^{1-\alpha}-(k-1)^{1-\alpha},
$$
we get 
\begin{equation*}
\begin{split}
\frac{1-\alpha}{k^{\alpha}l_k}
&\leq \left(\frac{k^{1-\alpha}}{l_k}-\frac{(k-1)^{1-\alpha}}{l_{k-1}}\right)
+(k-1)^{1-\alpha}\left(\frac{1}{l_{k-1}}-\frac{1}{l_k}\right)\\
&=\left(\frac{k^{1-\alpha}}{l_k}-\frac{(k-1)^{1-\alpha}}{l_{k-1}}\right)
+\frac{(k-1)^{1-\alpha}(\alpha_k-\alpha)}{(k+\alpha)l_k}.
\end{split}
\end{equation*}
Combining \eqref{eq:asymp-inv} with the relation
$$\frac{(k-1)^{1-\alpha}}{k+\alpha}-\frac{1}{k^{\alpha}}\sim -\frac{1-\alpha}{k^{\alpha+1}},$$
we have for some $c_4>0$,
$$\Gamma(\alpha+1)\frac{(k-1)^{1-\alpha}}{k+\alpha}-\frac{1}{g_k}
\sim -\frac{c_4}{k^{\alpha+1}}.$$
Therefore, we obtain for some $c_5>0$  and for any $n\geq 1$, 
$$
\Gamma(\alpha+1)(1-\alpha)\sum_{k=1}^n\frac{1}{k^{\alpha}l_k}
\leq c_5+\Gamma(\alpha+1)\frac{n^{1-\alpha}}{l_n}
+\sum_{k=1}^{n}\frac{\alpha_k-\alpha}{a_k}.
$$
In the same way, we have for some $c_6>0$ and for any $n\geq 1$,
$$\Gamma(\alpha+1)(1-\alpha)\sum_{k=1}^n\frac{1}{k^{\alpha}l_k}
\geq -c_6+\Gamma(\alpha+1)\frac{n^{1-\alpha}}{l_n}
+\sum_{k=1}^{n}\frac{\alpha_k-\alpha}{a_k}.$$
Hence there exists a positive constant $c_7$ such that for any $n\geq 1$,
$$|{\rm (II)}|
\leq \frac{c_7a_{n-1}}{n}.$$

By the triangle inequality and \eqref{eq:prod-asymp},
\begin{equation*}
\begin{split}
|{\rm (III)}|
&=\left|\frac{\Gamma(\alpha+1)g_{n-1}}{n^{\alpha}}\frac{l_{n-1}}{l_n}-1\right|\\
&\leq \frac{\Gamma(\alpha+1)}{n^{\alpha}}\frac{l_{n-1}}{l_n}\left|g_{n-1}-\frac{(n-1)^{\alpha}}{\Gamma(\alpha+1)}\right|
+\frac{l_{n-1}}{l_n}\left\{1-\left(1-\frac{1}{n}\right)^{\alpha}\right\}+\left|\frac{l_{n-1}}{l_n}-1\right|\\
&\leq \frac{c_8}{n}.
\end{split}
\end{equation*}
Combining the estimates of ${\rm (I)}$, ${\rm (II)}$, ${\rm (III)}$ above with \eqref{eq:sum-small},
we have (1).

We next prove (2). By Lemma \ref{lem:moment} (1), 
\begin{equation}\label{eq:exp-mean}
\begin{split}
&\frac{E[S_n]}{n}-1
=\frac{a_{n-1}}{n}\left(2q-1+\sum_{k=1}^{n-1}\frac{(1-\alpha_k)\varepsilon_k}{a_k}\right)-1\\
&=\frac{a_{n-1}}{n}(2q-1)-\frac{a_{n-1}}{n}\sum_{k=1}^{n-1}\frac{(1-\alpha_k)(1-\varepsilon_k)}{a_k}
+\left((1-\alpha)\frac{a_{n-1}}{n}\sum_{k=1}^{n-1}\frac{1}{a_k}-1-\frac{a_{n-1}}{n}\sum_{k=1}^{n-1}\frac{\alpha_k-\alpha}{a_k}\right)\\
&=A_n-B_n+C_n.
\end{split}
\end{equation}
Since $\alpha+\rho<1$ by assumption, we have 
$$\lim_{n\rightarrow\infty}\frac{A_n}{1-\varepsilon_n}=0, \quad 
\lim_{n\rightarrow\infty}\frac{B_n}{1-\varepsilon_n}=\frac{1-\alpha}{1-(\alpha+\rho)}.$$
By (1), we also get 
$$\lim_{n\rightarrow\infty}\frac{C_n}{1-\varepsilon_n}=0$$
so that 
$$\lim_{n\rightarrow\infty}\frac{1}{1-\varepsilon_n}\left(\frac{E[S_n]}{n}-1\right)
=-\frac{1-\alpha}{1-(\alpha+\rho)}.$$
On the other hand, 
since $\rho<1/2$ by assumption, 
we can apply Theorem \ref{thm:lln-0} with $r_n=n(1-\varepsilon_n)$ 
to obtain 
\begin{equation}\label{eq:lim-rate}
\lim_{n\rightarrow\infty}\frac{1}{1-\varepsilon_n}\left(\frac{S_n-E[S_n]}{n}\right)=0, \quad \text{$P$-a.s.}
\end{equation}
These two expressions above imply the desired assertion.
\end{proof}

\begin{proof}[Proof of Theorem {\rm \ref{thm:quad-var}}]
Since $X_n^2=1$ for any $n\in {\mathbb N}$, we have by Lemma \ref{lem:moment} (1),
\begin{equation}\label{eq:cond-var}
\begin{split}
&E[(X_j-E[X_j\mid {\cal F}_{j-1}])^2\mid {\cal F}_{j-1}]
=1-\left(\frac{\alpha_{j-1}}{j-1}S_{j-1}+(1-\alpha_{j-1})\varepsilon_{j-1}\right)^2\\
&=\left\{\alpha_{j-1}\left(1-\frac{S_{j-1}}{j-1}\right)+(1-\alpha_{j-1})(1-\varepsilon_{j-1})\right\}
\left\{1+\alpha_{j-1}\frac{S_{j-1}}{j-1}+(1-\alpha_{j-1})\varepsilon_{j-1}\right\}.
\end{split}
\end{equation}
Then by Lemma \ref{lem:lln-order} (2) and Example \ref{exam:lln} (i),
\begin{equation*}
\begin{split}
\alpha_{j-1}\left(1-\frac{S_{j-1}}{j-1}\right)+(1-\alpha_{j-1})(1-\varepsilon_{j-1})
&\sim \frac{\alpha(1-\alpha)}{1-(\alpha+\rho)}(1-\varepsilon_j)+(1-\alpha)(1-\varepsilon_j)\\
&=c_{\alpha,\rho}(1-\varepsilon_j)
\end{split}
\end{equation*}
and 
$$1+\alpha_{j-1}\frac{S_{j-1}}{j-1}+(1-\alpha_{j-1})\varepsilon_{j-1}\sim 2\sim 1+\varepsilon_j.$$
We thus have as $j\rightarrow\infty$,
\begin{equation}\label{eq:sim-quad}
E[(X_j-E[X_j\mid {\cal F}_{j-1}])^2\mid {\cal F}_{j-1}]
\sim c_{\alpha,\rho}(1-\varepsilon_j^2).
\end{equation}

If $\sum_{n=1}^{\infty}(1-\varepsilon_n)/a_n^2=\infty$, 
then \eqref{eq:sim-quad} implies that as $n\rightarrow\infty$, 
$$\sum_{j=1}^nE[d_j^2\mid {\cal F}_{j-1}]
=\sum_{j=1}^n \frac{1}{a_j^2}E[(X_j-E[X_j\mid {\cal F}_{j-1}])^2\mid {\cal F}_{j-1}]
\sim c_{\alpha,\rho}\sum_{j=1}^n\frac{1-\varepsilon_j^2}{a_j^2}=c_{\alpha,\rho}v_n,$$
that is, \eqref{eq:quad-asymp} holds.
Hence we get (1) by \cite[Lemmas 3.4 and 3.5]{JJQ08}, 
which follow from \cite[Corollary 3.1, Theorems 4.7 and 4.8]{HH80}.

On the other hand, if $\sum_{n=1}^{\infty}(1-\varepsilon_n)/a_n^2<\infty$, 
then by \eqref{eq:cond-var} and Lemma \ref{lem:lln-order} (2), 
we have for some random positive constant $C$, 
$$\sum_{n=1}^{\infty}E[d_n^2\mid {\cal F}_{n-1}]
\leq C\sum_{n=1}^{\infty}\frac{1-\varepsilon_n}{a_n^2}<\infty.$$
Hence by \cite[Theorem 2.15]{HH80}, 
$M_{\infty}:=\sum_{n=1}^{\infty}(M_n-M_{n-1})$ exists $P$-a.s.\ and 
$M_n\rightarrow M_{\infty}$ in $L^2(P)$.
This yields
$$\lim_{n\rightarrow\infty}\frac{S_n-E[S_n]}{a_n}
=\lim_{n\rightarrow\infty}M_n=M_{\infty}, \ \text{$P$-a.s.\ and in $L^2(P)$}$$
and 
$$E[M_{\infty}]=\lim_{n\rightarrow\infty}E[M_n]=0, 
\quad E[M_{\infty}^2]=\sum_{n=1}^{\infty}E[(M_n-M_{n-1})^2]\in (0,\infty).$$
The proof is complete.
\end{proof}

\begin{rem}\rm 
In the proof of Theorem \ref{thm:quad-var}, we used the condition $\rho<1/2$ in Assumption \ref{assum:alpha} 
only for the validity of \eqref{eq:lim-rate}; 
the rest of the proof is still valid for $\rho<1-\alpha$. 
At present, we do not know if the statement of Theorem \ref{thm:quad-var} is true or not 
for $\rho<1-\alpha$.
\end{rem}

\subsection{Growth exponent}
In this subsection, we discuss again the growth exponent of $\{S_n\}_{n=0}^{\infty}$ 
to interpolate Corollary \ref{cor:lln-0} and Theorem \ref{thm:conv-limit}. 
Let $w_n=\sum_{k=1}^n1/a_k^2$ and $v_n=\sum_{k=1}^n(1-\varepsilon_k^2)/a_k^2$. 
As a corollary of Theorems \ref{thm:fluct-0} and \ref{thm:quad-var}, 
we have 
\begin{cor}\label{cor:conv-limit-1}
\begin{enumerate}
\item[{\rm (1)}] Assume that $\varepsilon=0$, $\sum_{n=1}^{\infty}1/a_n^2=\infty$ 
and the following limit exists as a finite value{\rm :}
\begin{equation}\label{eq:kappa}
\mu:=\lim_{n\rightarrow\infty}\frac{1-\alpha}{\sqrt{w_n}}\sum_{k=1}^n\frac{\varepsilon_k}{a_k}. 
\end{equation}
Then 
\begin{equation}\label{eq:growth-dist}
\frac{S_n}{a_n\sqrt{w_n}}\rightarrow N(\mu,1) \quad \text{in distribution}.
\end{equation}
\item[{\rm (2)}] If 
\begin{equation}\label{eq:var-dom}
\sum_{n=1}^{\infty}\frac{(1-\alpha_n)\varepsilon_n}{a_n}<\infty,
\end{equation}
then $E[S_n]\sim c_*a_n$ with 
$$c_*=2q-1+\sum_{n=1}^{\infty}\frac{(1-\alpha_n)\varepsilon_n}{a_n}.$$
Moreover, if $\sum_{n=1}^{\infty}1/a_n^2<\infty$, 
or if Assumption {\rm \ref{assum:alpha}} is fulfilled and  $\sum_{n=1}^{\infty}(1-\varepsilon_n)/a_n^2<\infty$, 
then 
\begin{equation}\label{eq:mtg-conv}
\frac{S_n}{a_n}\rightarrow M_{\infty}+c_* \quad \text{$P$-a.s.\ and in $L^2(P)$.}
\end{equation}
\end{enumerate}
\end{cor}

\begin{proof}
Under the assumption of (1), Lemma \ref{lem:moment} (1) yields 
$$\lim_{n\rightarrow\infty}\frac{E[S_n]}{a_n\sqrt{w_n}}=\mu.$$ 
Then (1) follows by Theorem \ref{thm:fluct-0}.
If \eqref{eq:var-dom} is fulfilled, then by Lemma \ref{lem:moment} (1), $E[S_n]\sim c_*a_n$.   
Therefore, the proof is complete by Theorems \ref{thm:fluct-0} (2) and \ref{thm:quad-var}. 
\end{proof}

Note that if $0<\varepsilon\leq 1$ and $\sum_{n=1}^{\infty}1/a_n^2=\infty$, 
then $0\leq \alpha\leq 1/2$ so that $\mu=\infty$ in \eqref{eq:kappa}.

\begin{exam}\label{exam:l2-lln}\rm 
Let $0<\varepsilon\leq 1$. 
If $0\leq \alpha<1$, or if $\alpha=1$ and $\sum_{n=1}^{\infty}{(1-\alpha_n)}/n=\infty$, 
then by Example \ref{exam:lln}, $S_n/n\rightarrow \varepsilon,$ $P$-a.s.\ and in $L^2(P)$.
If $\alpha=1$ and  $\sum_{n=1}^{\infty}{(1-\alpha_n)}/n<\infty$, 
then $\sum_{n=1}^{\infty}(1-\alpha_n)/a_n<\infty$ by Lemma \ref{lem:series} (5). 
Hence \eqref{eq:mtg-conv} holds by Corollary \ref{cor:conv-limit-1} (2). 
\end{exam}

\begin{exam}\label{exam:l2-lln-1}\rm
Let $\varepsilon=0$ and $0\leq\alpha<1$. 
Assume that the sequence $\{\varepsilon_n\}_{n=1}^{\infty}$ is regularly varying 
with index $-\rho$ for some $\rho\geq 0$. 

\noindent
(i) Let $\alpha+\rho<1$. 
We first assume that 
\begin{itemize}
\item$1/2<\rho<1-\alpha$.
\end{itemize}
Since $\mu=0$, we have by Corollary \ref{cor:conv-limit-1} (1),
$$\frac{S_n}{\sqrt{n}}\rightarrow 
N\left(0,\frac{1}{1-2\alpha}\right) 
\quad \text{in distribution}. $$ 
We next assume that 
\begin{itemize}
\item $0<\alpha<1/2$ and $\rho=1/2$, and  
the limit  
$\mu_0:=\lim_{n\rightarrow\infty}\sqrt{n}\varepsilon_n$ exists as a finite value. 
\end{itemize}
Let $\mu$ be as in  \eqref{eq:kappa}. Since
$$\mu=\frac{\mu_0\sqrt{1-2\alpha}(1-\alpha)}{1-(\alpha+\rho)}, $$
Corollary \ref{cor:conv-limit-1} (1) yields 
$$\frac{S_n}{\sqrt{n}}\rightarrow 
N\left(\frac{\mu_0\sqrt{1-2\alpha}(1-\alpha)}{1-(\alpha+\rho)},\frac{1}{1-2\alpha}\right) 
\quad \text{in distribution}. $$ 

\noindent
(ii) Let $\alpha+\rho>1$.  Then \eqref{eq:var-dom} is valid. 
Moreover, if $\sum_{n=1}^{\infty}{1/a_n^2}=\infty$, 
then $\mu=0$ so that by Corollary \ref{cor:conv-limit-1} (1),
$$\frac{S_n}{a_n\sqrt{w_n}}\rightarrow N(0,1) \quad \text{in distribution}.$$
On the other hand, if $\sum_{n=1}^{\infty}{1/a_n^2}<\infty$, 
then Corollary \ref{cor:conv-limit-1} (2) yields  
\eqref{eq:mtg-conv}. 
\end{exam}

\begin{exam}\label{exam:border}\rm 
For some constants $\eta$, $\kappa$ and $\theta>0$, let 
$$\varepsilon_n=\frac{(\log n)^{\eta-1}}{n^{1-\alpha}}, 
\quad \alpha_n=\alpha+\frac{\kappa}{(\log n)^{\theta}} \ (n\geq 2).$$
If $0\leq \alpha<1/2$, then $\sum_{n=1}^{\infty}1/a_n^2=\infty$ and $\mu=0$.
Hence by Corollary \ref{cor:conv-limit-1} (1), 
$$\frac{S_n}{\sqrt{n}}\rightarrow N\left(0,\frac{1}{1-2\alpha}\right) \quad \text{in distribution}.$$
In what follows, we assume that $1/2\leq \alpha<1$. 
See Examples \ref{exam:order-l-a-1} and \ref{exam:order-l-a-2} 
for the asymptotic behaviors of $a_n$ and $a_n\sqrt{w_n}$.

\noindent
(i) Let $0<\theta<1$ and $\kappa\ne 0$. 
Then Corollary \ref{cor:conv-limit-1} implies the following: if
$\alpha=1/2$, $\kappa<0$, and $\eta+\theta/2<1$, then  
$$\frac{S_n}{a_n\sqrt{w_n}}\rightarrow N(0,1)\quad \text{in distribution}{\rm ;}$$
if $\alpha=1/2$, $\kappa<0$ and $\eta+\theta/2=1$,
then
$$\frac{S_n}{a_n\sqrt{w_n}}\rightarrow N\left(\frac{1}{\sqrt{-2\kappa}},1\right) \quad \text{in distribution}.$$
On the other hand, if $1/2\leq \alpha<1$ and $\kappa>0$,
then by Corollary \ref{cor:conv-limit-1} (2), 
\eqref{eq:mtg-conv} holds. 

\noindent
(ii) Let $\theta=1$. 
Then Corollary \ref{cor:conv-limit-1} implies the following: if
$\alpha=1/2$, $\kappa<1/2$, and $\eta<1/2$, then  
$$\frac{S_n}{a_n\sqrt{w_n}}\rightarrow N(0,1)\quad \text{in distribution}{\rm ;}$$
if $\alpha=1/2$, $\kappa<1/2$ and $\eta=1/2$,
then
$$\frac{S_n}{a_n\sqrt{w_n}}\rightarrow N\left(\frac{1}{\sqrt{1-2\kappa}},1\right) \quad \text{in distribution}.$$
On the other hand, if 
\begin{itemize}
\item $\alpha=1/2$ and $\kappa>\eta \vee (1/2)$, or 
\item $1/2<\alpha<1$ and $\eta<0$,
\end{itemize}
then  by Corollary \ref{cor:conv-limit-1} (2), \eqref{eq:mtg-conv} holds. 

\noindent
(iii) Let $\theta>1$ or $\kappa=0$. 
Then by Corollary \ref{cor:conv-limit-1}, if
$\alpha=1/2$ and $\eta<1/2$, then  
$$\frac{S_n}{a_n\sqrt{w_n}}\rightarrow N(0,1)\quad \text{in distribution}{\rm ;}$$
if $\alpha=1/2$ and $\eta=1/2$,
then
$$\frac{S_n}{a_n\sqrt{w_n}}\rightarrow N(1,1) \quad \text{in distribution}.$$
On the other hand, if 
$1/2<\alpha<1$ and $\eta<0$, 
then \eqref{eq:mtg-conv} holds by Corollary \ref{cor:conv-limit-1} (2).

Combining (i)--(iii) above with Example \ref{exam:0-case}, 
we can draw graphs
describing the phase transition on the growth exponent of $\{S_n\}_{n=0}^{\infty}$ 
in terms of $\kappa$ and $\eta$. 
We here present the graphs for $\alpha=1/2$ in Figures \ref{figure:a}, \ref{figure:b} and \ref{figure:c}. 
The colors of the graphs correspond to the equation numbers respectively as we see in Figures \ref{figure:a}.
See Example \ref{exam:order-l-a-2} for the behavior of $a_n\sqrt{w_n}$ in (i)--(iii).

\begin{figure}[H]
\begin{tabular}{ccc}
\begin{minipage}[t]{0.32\hsize}
\includegraphics{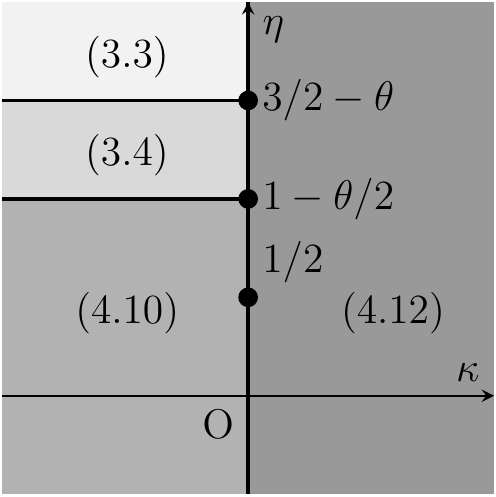}
\caption{$0<\theta<1$ and $\alpha=1/2$}\label{figure:a}
\end{minipage}
\begin{minipage}[t]{0.32\hsize}
\includegraphics{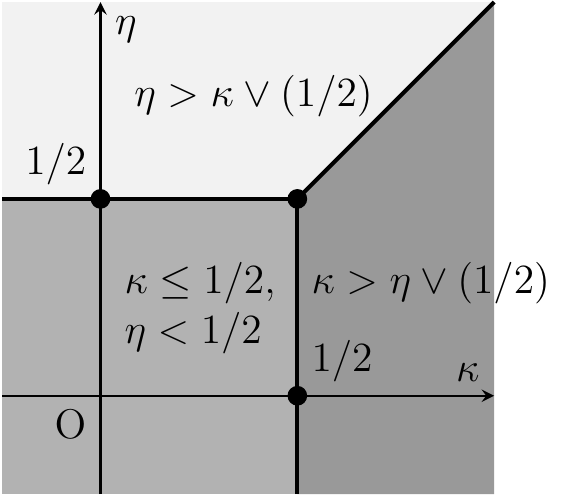}
\caption{$\theta=1$ and $\alpha=1/2$}\label{figure:b}
  \end{minipage} 
\begin{minipage}[t]{0.32\hsize}
\includegraphics{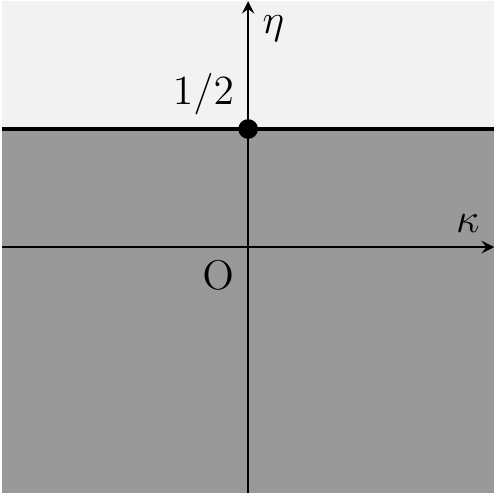}
\caption{$\theta>1$ and $\alpha=1/2$}\label{figure:c}
  \end{minipage}
\end{tabular}
\end{figure}
\end{exam}

\begin{exam}\rm 
For some positive constants $\eta$, $\kappa$ and $\theta$, let 
$$\varepsilon_n=\frac{1}{(\log n)^{\eta}}, \quad \alpha_n=1-\frac{\kappa}{(\log n)^{\theta}} \ (n\geq 2).$$
Let 
$$r_n=a_n\sum_{k=1}^n\frac{(1-\alpha_k)\varepsilon_k}{a_k}.$$
We saw in Example \ref{exam:lln-2} that if $0<\theta<1$, or if $\theta=1$ and $\kappa\geq \eta$,  then 
$S_n/r_n\rightarrow 1$, $P$-a.s.\ and in $L^2(P)$.
On the other hand, if $\theta=1$ and $\eta>\kappa$, or if $\theta>1$, 
then $\sum_{n=1}^{\infty}(1-\alpha_n)\varepsilon_n/a_n<\infty$ 
by Example \ref{exam:order-l-a-1}. 
Since $\sum_{n=1}^{\infty}1/a_n^2<\infty$, 
Corollary \ref{cor:conv-limit-1} (2) yields
\eqref{eq:mtg-conv}. 
\end{exam}

\subsection{Gaussian fluctuation around the random drift}
Let $z_n=\sum_{k=n}^{\infty} 1/a_k^2$ and $t_n=\sum_{k=n}^{\infty}(1-\varepsilon_k^2)/a_k^2$. 
Let $\psi(t)=\sqrt{2t\log|\log t|}$. 
\begin{thm}\label{thm:fluct}
\begin{enumerate}
\item[{\rm (1)}] If $0\leq \varepsilon<1$ and $\sum_{n=1}^{\infty}1/a_n^2<\infty$,
then
$$
\frac{S_n-E[S_n]-a_nM_{\infty}}{a_n\sqrt{z_n}}\rightarrow N(0,1-\varepsilon^2) \quad  \text{in distribution}
$$
and 
$$
\limsup_{n\rightarrow\infty}\pm \frac{S_n-E[S_n]-a_nM_{\infty}}{a_n\psi(z_n)}=\sqrt{1-\varepsilon^2}, \quad \text{$P$-a.s.}
$$
\item[{\rm (2)}] Assume that $\varepsilon=1$ and 
\begin{enumerate}
\item[{\rm (a)}] $\sum_{n=1}^{\infty}1/a_n^2<\infty$ and the sequence $\{1-\varepsilon_n\}_{n=1}^{\infty}$ 
is regularly variyng with index $-\rho$ for some $\rho\geq 0$, or 
\item[{\rm (b)}] Assumption {\rm \ref{assum:alpha}} is fulfilled 
and $\sum_{n=1}^{\infty}(1-\varepsilon_n)/a_n^2<\infty$.
\end{enumerate}
Then 
\begin{equation}\label{eq:fluct-clt}
\frac{S_n-E[S_n]-a_nM_{\infty}}{a_n\sqrt{t_n}}\rightarrow N(0,c_{\alpha,\rho}) \quad  \text{in distribution}
\end{equation}
and 
\begin{equation}\label{eq:fluct-lil}
\limsup_{n\rightarrow\infty}\pm \frac{S_n-E[S_n]-a_nM_{\infty}}{a_n\psi(t_n)}
=\sqrt{c_{\alpha,\rho}}, \quad \text{$P$-a.s.}
\end{equation}
\end{enumerate}
\end{thm}

\begin{proof}
We show (2) under the condition (b) only 
because the rest of the assertions is proved in a similar  way. 
We take the same approach as the proof of \cite[Theorem 3]{KT19}. 
Let Assumption {\rm \ref{assum:alpha}} hold and $\sum_{n=1}^{\infty}(1-\varepsilon_n)/a_n^2<\infty$.
Define 
$$V_n^2=\sum_{k=n}^{\infty}E[d_k^2\mid {\cal F}_{k-1}], \quad 
\sigma_n^2=\sum_{k=n}^{\infty}E[d_k^2].$$
Then \eqref{eq:sim-quad} yields  
$V_n^2\sim c_{\alpha,\rho}t_n$, $P$-a.s.\ 
and 
$\sigma_n^2\sim c_{\alpha,\rho}t_n$ 
so that $V_n^2/\sigma_n^2\rightarrow 1$ as $n\rightarrow\infty$, $P$-a.s.

We first prove \eqref{eq:fluct-clt}. By assumption, $\rho+2\alpha\geq 1$ and 
the sequence $\{(1-\varepsilon_n)/a_n^2\}_{n=1}^{\infty}$ is regularly varying with index $-(\rho+2\alpha)$.  
Then by \cite[Lemma 1.5.9 b]{BGT87}, we have  for some positive constant $c_1$,  
\begin{equation}\label{eq:t-lower}
a_n^2t_n=n(1-\varepsilon_n)\left(\frac{n(1-\varepsilon_n)}{a_n^2}\right)^{-1}
\sum_{k=n}^{\infty}\frac{1-\varepsilon_k^2}{a_k^2}\geq  c_1n(1-\varepsilon_n).
\end{equation}
Since  $|d_n|\leq 2/a_n$ by \eqref{eq:difference}, 
we  obtain by \eqref{eq:t-lower} and Assumption \ref{assum:alpha} (ii), 
\begin{equation}\label{eq:fourth}
\sum_{n=1}^{\infty}\frac{1}{\sigma_n^4}E[d_n^4\mid {\cal F}_{n-1}]
\leq c_2\sum_{n=1}^{\infty}\frac{1}{(a_n^2t_n)^2}
\leq c_3\sum_{n=1}^{\infty}\frac{1}{n^2(1-\varepsilon_n)^2}<\infty.
\end{equation}

Let $K_0=0$ and 
$$K_n=\sum_{k=1}^n \frac{1}{\sigma_k^2}(d_k^2-E[d_k^2\mid {\cal F}_{k-1}]) \quad (n\geq 1). $$
Then $\{K_n\}_{n=0}^{\infty}$ is a martingale. 
Let 
$$\Delta K_n=K_n-K_{n-1}=\frac{d_n^2-E[d_n^2\mid {\cal F}_{n-1}]}{\sigma_n^2}.$$
Since 
$$(\Delta K_n)^2\leq \frac{c_4}{\sigma_n^4}(d_n^4+E[d_n^4\mid{\cal F}_{n-1}]),$$
we have by \eqref{eq:fourth}, 
$$\sum_{n=1}^{\infty}(\Delta K_n)^2
\leq c_5\sum_{n=1}^{\infty}\frac{1}{\sigma_n^4}(d_n^4+E[d_n^4\mid{\cal F}_{n-1}])<\infty.$$
Hence by \cite[Theorem 2.15]{HH80}, the series 
$$\sum_{k=1}^{\infty}\frac{1}{\sigma_k^2}(d_k^2-E[d_k^2 \mid {\cal F}_{k-1}])$$
is almost surely convergent. Then by Kronecker's lemma, 
$$\lim_{n\rightarrow\infty}\frac{1}{\sigma_n^2}\sum_{k=n}^{\infty}(d_k^2-E[d_k^2\mid {\cal F}_{k-1}])=0, \quad \text{$P$-a.s.}$$ 
and thus 
\begin{equation}\label{eq:as-variance}
\frac{1}{\sigma_n^2}\sum_{k=n}^{\infty}d_k^2
=\frac{1}{\sigma_n^2}\sum_{k=n}^{\infty}(d_k^2-E[d_k^2\mid {\cal F}_{k-1}])
+\frac{1}{\sigma_n^2}\sum_{k=n}^{\infty}E[d_k^2\mid {\cal F}_{k-1}]\rightarrow 1, 
\quad \text{$P$-a.s.}
\end{equation}

Since the sequence $\{a_n\}_{n=1}^{\infty}$ is increasing by definition 
and $|d_n|\leq 2/a_n$, we have 
$$\sup_{k\geq n}(d_k^2) \leq \frac{4}{a_n^2}\rightarrow 0  \ (n\rightarrow\infty)$$
so that by \eqref{eq:t-lower},
$$\frac{E[\sup_{k\geq n}(d_k^2)]}{\sigma_n^2}\leq \frac{4}{a_n^2\sigma_n^2}
\leq \frac{c_6}{a_n^2t_n}
\leq \frac{c_7}{n(1-\varepsilon_n)}\rightarrow 0 \ (n\rightarrow\infty).$$ 
Combining this relation with \eqref{eq:as-variance}, we obtain by \cite[Corollary 3.5]{HH80}, 
$$\frac{M_{\infty}-M_n}{\sigma_{n+1}}\rightarrow N(0,1) \quad  \text{in distribution}.$$
As $M_n=(S_n-E[S_n])/a_n$ and $\sigma_n\sim \sqrt{c_{\alpha,\rho}}\sqrt{t_n}$, we have \eqref{eq:fluct-clt}.

We next prove \eqref{eq:fluct-lil}. 
For any $\eta>0$, we have as for \eqref{eq:fourth}, 
$$\sum_{n=1}^{\infty}\frac{1}{\sigma_n}E[|d_n|; |d_n|>\eta \sigma_n]
\leq \frac{1}{\eta^3}\sum_{n=1}^{\infty}\frac{1}{\sigma_n^4}E[d_n^4]<\infty$$
and 
$$\sum_{n=1}^{\infty}\frac{1}{\sigma_n^4}E[d_n^4; |d_n|\leq \eta \sigma_n]
\leq \sum_{n=1}^{\infty}\frac{1}{\sigma_n^4}E[d_n^4]<\infty.$$
By these calculations with \eqref{eq:as-variance}, 
we can apply \cite[Theorem 4.9]{HH80} to show that 
$$\limsup_{n\rightarrow\infty}\pm\frac{M_{\infty}-M_n}{\psi(\sigma_n^2)}=1, \quad \text{$P$-a.s.}$$
This implies \eqref{eq:fluct-lil}. 
\end{proof}

\subsection{Examples}
\begin{exam}\label{exam:clt-1}\rm 
Assume that $\varepsilon\in [0,1)$. Then the Gaussian fluctuation results in \cite{KT19} 
remain true.

\noindent
(i) If $0\leq \alpha<1/2$, then by Theorem \ref{thm:fluct-0},
$$\frac{S_n-E[S_n]}{\sqrt{n}}
\rightarrow N\left(0,\frac{1-\varepsilon^2}{1-2\alpha}\right) \quad \text{in distribution}$$
and 
$$\limsup_{n\rightarrow\infty}\pm \frac{S_n-E[S_n]}{\sqrt{2n\log\log n}}=\sqrt{\frac{1-\varepsilon^2}{1-2\alpha}}, \quad \text{$P$-a.s.}$$

\noindent
(ii) If $1/2<\alpha\leq 1$, then  by Theorem \ref{thm:fluct},
$$\frac{S_n-E[S_n]-a_nM_{\infty}}{\sqrt{n}}\rightarrow N\left(0,\frac{1-\varepsilon^2}{2\alpha-1}\right) \quad \text{in distribution}$$
and 
$$\limsup_{n\rightarrow\infty}
\pm \frac{S_n-E[S_n]-a_nM_{\infty}}{\sqrt{2n\log\log n}}
=\sqrt{\frac{1-\varepsilon^2}{2\alpha-1}}, \quad \text{$P$-a.s.}$$

\noindent
(iii) 
For some constants $\kappa$ and $\theta>0$, let 
$$\alpha_n=\frac{1}{2}+\frac{\kappa}{(\log n)^{\theta}} \quad (n\geq 2).$$
Then by Example \ref{exam:order-l-a-2}, 
if $0<\theta<1$ and $\kappa>0$, or if $\theta=1$ and $\kappa>1/2$, 
then the assertions of  Theorem \ref{thm:fluct-0} (2) and Theorem \ref{thm:fluct} (1) are valid; 
otherwise, that of Theorem \ref{thm:fluct-0} (1) holds.
See Example \ref{exam:order-l-a-2} 
for the calculations of the scaling factors $a_n\sqrt{w_n}$ and $a_n\sqrt{z_n}$ 
in Theorem \ref{thm:fluct-0} (1) and Theorem \ref{thm:fluct} (1).
Note that if $0<\theta<1$, then 
the asymptotic behaviors of $a_n\sqrt{w_n}$ and $a_n\sqrt{z_n}$ are dependent of the value of $\theta$. 
\end{exam}

\begin{exam}\label{exam:clt-2}\rm 
Let  $\varepsilon=1$. We assume that the sequence $\{1-\varepsilon_n\}_{n=1}^{\infty}$ 
is regularly varying of index $-\rho$ for some $\rho\geq 0$.

\noindent
(i) Let $1/2<\alpha\leq 1$. Since $\sum_{n=1}^{\infty}1/a_n^2<\infty$, 
we have the assertion in  Theorem \ref{thm:fluct} (2) with 
\begin{equation}\label{eq:t-asymp}
t_n\sim \frac{1}{2\alpha+\rho-1}\frac{n(1-\varepsilon_n)}{a_n^2}, 
\quad a_n\sqrt{t_n}\sim \sqrt{\frac{n(1-\varepsilon_n)}{2\alpha+\rho-1}}.
\end{equation}

\noindent
(ii) Let $0\leq \alpha<1/2$. Then $\sum_{n=1}^{\infty}1/a_n^2=\infty$. 
If $0<\rho<(1-2\alpha)\wedge (1/2)$, then the assertion of Theorem \ref{thm:quad-var} (1) holds with 
$$v_n\sim \frac{1}{1-(2\alpha+\rho)}\frac{n(1-\varepsilon_n)}{a_n^2}, 
\quad a_n\sqrt{v_n}\sim \sqrt{\frac{n(1-\varepsilon_n)}{1-(2\alpha+\rho)}}.$$
On the other hand, if $1-2\alpha<\rho<1/2$, 
then the assertion of Theorem \ref{thm:fluct} (2) holds with 
\eqref{eq:t-asymp}. 

We here assume that $\rho=1-2\alpha$ and $0<\rho<1/2$, that is, $1/4<\alpha<1/2$. 
Then the scaling factors in the Gaussian fluctuations 
depend on the convergence rates of $\{\varepsilon_n\}_{n=1}^{\infty}$ and $\{\alpha_n\}_{n=1}^{\infty}$. 
For instance, we first take for some constants $\eta$, $\kappa$ and $\theta>0$,
$$\varepsilon_n=1-\frac{(\log n)^{\eta-1}}{n^{1-2\alpha}}, 
\quad \alpha_n=\alpha+\frac{\kappa}{(\log n)^{\theta}} \quad (n\geq 2).$$
We assume one of  the following conditions:
\begin{itemize}
\item  $0<\theta<1$, $\kappa<0$, 
\item $\theta=1$ and $\eta\geq 2\kappa$, 
\item $\theta>1$ or $\kappa=0$, and $\eta\geq 0$.
\end{itemize}
Then the assertion of Theorem \ref{thm:quad-var} (1) holds; 
otherwise, that of Theorem \ref{thm:fluct} (2) holds. 
See Example \ref{exam:order-2} for the calculations of the scaling functions $a_n\sqrt{v_n}$ and $a_n\sqrt{t_n}$.

\noindent
(iii) Let $\alpha=1/2$. 
If $\rho\in (0,1/2)$, then  $\sum_{n=1}^{\infty}(1-\varepsilon_n)/a_n^2<\infty$ 
so that  the assertion of Theorem \ref{thm:fluct} (2) holds. 

If $\rho=0$, 
then the convergence rates of $\{\varepsilon_n\}_{n=1}^{\infty}$ and $\{\alpha_n\}_{n=1}^{\infty}$
affect the scaling factors in the Gaussian fluctuations. 
For instance, we first take for some constants $\eta>0$, $\kappa$ and $\theta>0$,
$$\varepsilon_n=1-\frac{1}{(\log n)^{\eta}}, \quad \alpha_n=\frac{1}{2}+\frac{\kappa}{(\log n)^{\theta}} \quad (n\geq 2).$$
We assume one of  the following conditions:
\begin{itemize}
\item  $0<\theta<1$, $\kappa<0$, 
\item $\theta=1$ and $0<\eta\leq 1-2\kappa$, 
\item $\theta>1$ or $\kappa=0$, and $0<\eta\leq 1$.
\end{itemize}
Then the assertion of Theorem \ref{thm:quad-var} (1)  is valid; 
otherwise, that of Theorem \ref{thm:fluct} (2) is valid. 
See Example \ref{exam:order-1} 
also for the calculations of the scaling factors $a_n\sqrt{v_n}$ and $a_n\sqrt{t_n}$ 
in Theorems \ref{thm:quad-var} (1) and \ref{thm:fluct} (2). 
\end{exam}

\appendix
\section{Appendix}\label{sect:append}
\subsection{Asymptotic properties of series}
The next lemma is  used for the proofs of Lemma \ref{lem:series} and Theorem \ref{thm:conv-limit}. 
\begin{lem}\label{lem:series-order}
Let $\{c_n\}_{n=1}^{\infty}$ be a nonnegative sequence such that 
$c_n\rightarrow 0$ as $n\rightarrow\infty$ and $\sum_{n=1}^{\infty}c_n=\infty$. 
Then 
\begin{equation}\label{eq:series-order-1}
\sum_{k=1}^nc_k\left(\sum_{l=1}^kc_l\right)=\sum_{k=1}^nc_k\left(\sum_{l=k}^nc_l\right)
\sim\frac{1}{2}\left(\sum_{k=1}^nc_k\right)^2
\end{equation}
and 
\begin{equation}\label{eq:series-order-2}
\sum_{k=1}^n c_k\exp\left(\sum_{l=1}^kc_l\right)\sim \exp\left(\sum_{k=1}^nc_k\right).
\end{equation}
\end{lem}

\begin{proof}
Since 
$$\sum_{k=1}^nc_k\left(\sum_{l=k}^nc_l\right)=\sum_{l=1}^nc_l\left(\sum_{k=1}^lc_k\right)
=\sum_{k=1}^nc_k\left(\sum_{l=1}^kc_l\right),$$
we have
$$\sum_{k=1}^nc_k\left(\sum_{l=k}^nc_l\right)
=\frac{1}{2}\left(\sum_{k=1}^nc_k\right)^2+\frac{1}{2}\sum_{k=1}^nc_k^2.$$
As $c_n\rightarrow 0 \ (n\rightarrow\infty)$, we obtain the first assertion.

Let $T_0=0$ and $T_n=\sum_{k=1}^nc_k \ (n\geq 1)$. Then 
$$\sum_{k=1}^n c_k\exp\left(\sum_{l=1}^kc_l\right)=\sum_{k=1}^n (T_{k}-T_{k-1})e^{T_k}$$
and 
$$\int_{T_{k-1}}^{T_k}e^x\,{\rm d}x\leq (T_{k}-T_{k-1})e^{T_k}\leq e^{c_k}\int_{T_{k-1}}^{T_k}e^x\,{\rm d}x.$$
As $c_n\rightarrow 0 \ (n\rightarrow\infty)$, we get the second assertion 
by elementary calculus.
\end{proof}

\subsection{Examples}
We present  the asymptotic behaviors of sequences 
related to the examples in Sections \ref{sect:lln} and \ref{sect:gauss}.

\begin{exam}\label{exam:order-l-a-1}\rm 
For some constants $\alpha>0$, $\kappa$ and $\theta>0$, let  
$$\alpha_n=\alpha+\frac{\kappa}{(\log n)^{\theta}} \quad (n\geq 2).$$
Let $C_*$ be the same constant as in Lemma \ref{lem:series} (4). 
If $0<\theta<1$, then the limit 
$$\lim_{n\rightarrow\infty}
\left(\sum_{k=2}^{n-1}\frac{1}{k(\log k)^{\theta}}
-\int_1^n\frac{1}{x(\log x)^{\theta}}\,{\rm d}x\right)=\lim_{n\rightarrow\infty}
\left(\sum_{k=2}^{n-1}\frac{1}{k(\log k)^{\theta}}
-\frac{(\log n)^{1-\theta}}{1-\theta}\right)$$ 
exists and takes a positive value so that 
$$\rho_n
\sim c_1\exp\left(\frac{\kappa}{1-\theta}(\log n)^{1-\theta}\right), 
\quad a_n \sim 
\frac{c_1C_*}{\Gamma(\alpha+1)}n^{\alpha}\exp\left(\frac{\kappa}{1-\theta}(\log n)^{1-\theta}\right).$$
For $\theta=1$, the same argument above applies and thus 
$$\rho_n\sim c_2(\log n)^{\kappa}, 
\quad a_n \sim \frac{c_2C_*}{\Gamma(\alpha+1)}n^{\alpha}(\log n)^{\kappa}.$$
For $\theta>1$, the sequence $\{\rho_n\}_{n=1}^{\infty}$ is convergent and 
$a_n\sim (\rho_{\infty}C_*/\Gamma(\alpha+1)) n^{\alpha}$ 
with $\rho_{\infty}=\lim_{n\rightarrow\infty}\rho_n$.

According to the calculations above, we see that if $\alpha=1/2$, then 
$$
\sum_{n=1}^{\infty}\frac{1}{a_n^2}<\infty 
\iff \text{$0<\theta<1$  and $\kappa>0$, or $\theta=1$ and $\displaystyle \kappa>\frac{1}{2}$}.
$$
On the other hand, since $\{a_n\}_{n=1}^{\infty}$ is a regularly varying sequence with index $-\alpha$, 
we have $\sum_{n=1}^{\infty}1/a_n^2=\infty$ for $\alpha\in (0,1/2)$, 
and $\sum_{n=1}^{\infty}1/a_n^2<\infty$ for $\alpha>1/2$.
\end{exam}

\begin{exam}\label{exam:order-l-a-2}\rm 
For some constants $\kappa$ and $\theta>0$, let 
$$\alpha_n=\frac{1}{2}+\frac{\kappa}{(\log n)^{\theta}} \quad (n\geq 2).$$
Let $w_n=\sum_{k=1}^n 1/a_k^2$ and $z_n=\sum_{k=n}^{\infty}1/a_k^2$. 
By the calculations in Example \ref{exam:order-l-a-1}, 
we have the following: 
Assume first that  $0<\theta<1$. 
If $\kappa<0$,  then  
$$w_n\sim 
\frac{c_1}{-2\kappa}(\log n)^{\theta}\exp\left(\frac{-2\kappa}{1-\theta}(\log n)^{1-\theta}\right), 
\quad a_n\sqrt{w_n}\sim \frac{1}{\sqrt{-2\kappa}}\sqrt{n(\log n)^{\theta}}.
$$
If $\kappa=0$, then  
$$w_n\sim c_2\log n, \quad
a_n\sqrt{w_n}\sim \sqrt{n\log n}.
$$
If $\kappa>0$, then 
$$z_n\sim \frac{c_3}{2\kappa}(\log n)^{\theta}\exp\left(-\frac{2\kappa}{1-\theta}(\log n)^{1-\theta}\right), \quad 
a_n\sqrt{z_n}\sim 
\sqrt{\frac{n(\log n)^{1-\theta}}{2\kappa}}.$$

Assume next that  $\theta=1$. 
If $\kappa<1/2$,  then  
$$w_n\sim 
\frac{c_4}{1-2\kappa}(\log n)^{1-2\kappa}, 
\quad a_n\sqrt{w_n}\sim \sqrt{\frac{n\log n}{1-2\kappa}}.
$$
If $\kappa=1/2$, then  
$$w_n\sim c_5\log\log n, \quad
a_n\sqrt{w_n}\sim \sqrt{n\log n \log\log n}.
$$
If $\kappa>1/2$, then 
$$
z_n\sim 
\frac{c_6}{2\kappa-1}\frac{1}{(\log n)^{2\kappa-1}}, \quad 
a_n\sqrt{z_n}\sim 
\sqrt{\frac{n\log n}{2\kappa-1}}.
$$

We finally assume that $\theta>1$. 
Then 
$$w_n\sim c_7\log n, \quad a_n\sqrt{w_n}\sim \sqrt{n\log n}.$$
\end{exam}

\begin{exam}\label{exam:order-2}\rm 
Let $0<\alpha<1/2$. For some constants $\eta$, $\kappa$ and $\theta>0$, let 
$$\varepsilon_n=1-\frac{(\log n)^{\eta-1}}{n^{1-2\alpha}}, \quad \alpha_n=\alpha+\frac{\kappa}{(\log n)^{\theta}}.$$ 

We first assume that $\theta=1$ or $\kappa=0$. If $\eta>2\kappa$, then 
$$v_n\sim \frac{c_1}{\eta-2\kappa}(\log n)^{\eta-2\kappa}, 
\quad a_n\sqrt{v_n}\sim \sqrt{\frac{n^{2\alpha}(\log n)^{\eta}}{\eta-2\kappa}}.$$
If $\eta=2\kappa$, then 
$$v_n\sim c_2\log \log n, \quad a_n\sqrt{v_n}\sim \sqrt{n^{2\alpha}\log\log n}. $$
If $\eta<2\kappa$, then 
$$t_n\sim \frac{c_3}{(2\kappa-\eta)(\log n)^{2\kappa-\eta}}, 
\quad a_n\sqrt{v_n}\sim \sqrt{\frac{n^{2\alpha}(\log n)^{\eta}}{2\kappa-\eta}}.$$
We next assume that $\theta>1$.  
Independently of the value of $\kappa$, 
the calculations above for $\theta=1$ remain valid by taking formally $\kappa=0$. 

We finally assume that $0<\theta<1$. If $\kappa>0$, then 
$$t_n\sim \frac{c_4(\log n)^{\eta+\theta}}{2\kappa}\exp\left(-\frac{2\kappa(\log n)^{1-\theta}}{1-\theta}\right),
\quad a_n\sqrt{t_n}\sim \sqrt{\frac{n^{2\alpha}(\log n)^{\eta+\theta}}{2\kappa}}.$$
If $\kappa<0$, then 
$$v_n\sim \frac{c_5(\log n)^{\eta+\theta}}{-2\kappa}\exp\left(\frac{-2\kappa(\log n)^{1-\theta}}{1-\theta}\right), 
\quad a_n\sqrt{v_n}\sim \sqrt{\frac{n^{2\alpha}(\log n)^{\eta+\theta}}{-2\kappa}}.$$
\end{exam}

\begin{exam}\label{exam:order-1}\rm 
For some constants $\eta>0$, $\kappa$ and $\theta>0$, let 
$$\varepsilon_n=1-\frac{1}{(\log n)^{\eta}}, \quad \alpha_n=\frac{1}{2}+\frac{\kappa}{(\log n)^{\theta}} \quad (n\geq 2).$$

We first assume that $\theta=1$ or $\kappa=0$. 
If $0<\eta<1-2\kappa$, then 
$$v_n\sim \frac{2c_1(\log n)^{1-(\eta+2\kappa)}}{1-(\eta+2\kappa)}, \quad 
a_n\sqrt{v_n}\sim
 \sqrt{\frac{2n(\log n)^{1-\eta}}{1-(\eta+2\kappa)}}.$$
 If $\eta=1-2\kappa$, then 
$$v_n\sim 2c_2\log\log n, \quad 
a_n\sqrt{v_n}\sim
\sqrt{2n(\log n)^{1-\eta}(\log\log n)}.$$
If $\eta>1-2\kappa$, then 
$$t_n\sim \frac{2c_3}{(\eta+2\kappa-1)(\log n)^{\eta+2\kappa-1}}, \quad 
a_n\sqrt{t_n}\sim \sqrt{\frac{2}{(\eta+2\kappa-1)n(\log n)^{\eta-1}}}.$$
We next assume that $\theta>1$. 
Then independently of the value of $\kappa$, the calculations above for $\theta=1$ remain valid by taking formally $\kappa=0$.

We finally assume that $0<\theta<1$. If $\kappa>0$, then 
$$t_n\sim \frac{c_4}{2\kappa} (\log n)^{\theta-\eta}\exp\left(-\frac{2\kappa}{1-\theta}(\log n)^{1-\theta}\right), 
\quad a_n\sqrt{t_n}\sim\sqrt{\frac{n(\log n)^{\theta-\eta}}{2\kappa}}.$$
If $\kappa<0$, then 
$$v_n\sim \frac{c_5}{-2\kappa}(\log n)^{\theta-\eta}\exp\left(\frac{-2\kappa}{1-\theta}(\log n)^{1-\theta}\right), 
\quad a_n\sqrt{v_n}\sim 
\sqrt{\frac{n(\log n)^{\theta-\eta}}{-2\kappa}}.$$
\end{exam}

\subsection*{Acknowledgments}
The author would like to thank Masato Takei 
for valuable discussions and helpful comments, 
which lead to the improvement of Assumption \ref{assum:alpha} and 
Lemma \ref{lem:lln-order}. 
He also kindly informed the author of the reference \cite{K09}.

\end{document}